\documentclass[a4paper, 12pt]{article}

\usepackage{lineno,hyperref,amsmath, amsthm, enumerate,amssymb, color}
\usepackage{cite}

\DeclareMathOperator{\Tor}{Tor}
\DeclareMathOperator{\Ann}{Ann}

\newtheorem{theorem}{Theorem}[section]

\newtheorem{proposition}[theorem]{Proposition}
\newtheorem{corollary}[theorem]{Corollary}
\theoremstyle{definition}
\newtheorem{definition}[theorem]{Definition}
\newtheorem{remark}[theorem]{Remark}
\newtheorem{example}[theorem]{Example}
\numberwithin{equation}{section}
\numberwithin{equation}{section}

\newcommand{\I}{{{I\hspace*{-0.2ex}}}}
\newcommand{\II}{{{I\hspace*{-0.5ex}I\hspace*{-0.2ex}}}}
\newcommand{\III}{{{I\hspace*{-0.5ex}I\hspace*{-0.5ex}I}}}

\newcommand{\Ip}{{{I\hspace*{-0.2ex}}^\prime}}
\newcommand{\IIp}{{{I\hspace*{-0.5ex}I\hspace*{-0.2ex}}^\prime}}

\renewcommand\footnotemark{}

\begin{document}

\date{}
\title{Self-Dual Abelian Codes in some Non-Principal\\ Ideal Group Algebras}

\author{Parinyawat Choosuwan, Somphong Jitman, and Patanee Udomkavanich}

\thanks{P. Choosuwan is with the Department of Mathematics and Computer Science, Faculty of Science,  Chulalongkorn University,   Bangkok 10330,  Thailand. Email: parinyawat.ch@gmail.com}

\thanks{S. Jitman is with the Department of Mathematics, Faculty of Science,  Silpakorn University,  Nakhon Pathom 73000,  Thailand.  Email: sjitman@gmail.com}  

\thanks{P. Udomkavanich is with the Department of Mathematics and Computer Science, Faculty of Science,  Chulalongkorn University,   Bangkok 10330,  Thailand. Email:  pattanee.u@chula.ac.th}

 \maketitle

\begin{abstract}
 	The main focus of this paper is    the complete    enumeration of self-dual abelian codes in non-principal ideal group algebras $\mathbb{F}_{2^k}[A\times \mathbb{Z}_2\times \mathbb{Z}_{2^s}]$ with respect to both the Euclidean and Hermitian inner products, where  $k$ and $s$ are   positive integers and  $A$ is an  abelian group  of odd order. Based on the well-known characterization of  Euclidean and Hermitian self-dual abelian codes,  we show  that  such enumeration  can be obtained in terms of   a suitable product of the number of cyclic codes,  the number of  Euclidean   self-dual cyclic codes,  and    the number of  Hermitian self-dual cyclic codes of length $2^s$ over some Galois extensions of  the ring \mbox{$\mathbb{F}_{2^k}+u\mathbb{F}_{2^k}$}, where $u^2=0$.  Subsequently, general results on the characterization and enumeration of cyclic codes and self-dual codes of length $p^s$ over     $\mathbb{F}_{p^k}+u\mathbb{F}_{p^k}$  are given.  
     Combining these results, the  complete enumeration  of   self-dual abelian codes in   $\mathbb{F}_{2^k}[A\times \mathbb{Z}_2\times \mathbb{Z}_{2^s}]$ is therefore obtained.

\end{abstract}



 
 \section{Introduction} 
 
 Information media, such as communication systems and storage devices of data, are not $100$ percent reliable in practice because of noise or other forms of introduced interference. The art of error correcting codes is a branch of Mathematics that has been introduced to deal with this problem since $1960$s. 
 Linear codes  with additional algebraic structures  and self-dual codes are   important families of  codes that have been extensively studied for both theoretical and practical reasons (see \cite{Be1967, Be1967_2,BS2011,Ch1992,DKL2000,JLLX2012,RS1992,NRS2006,JLS2013,S1993} and references therein).   
 Some major results on Euclidean self-dual cyclic codes have been established  in \cite{KZ2008}. In \cite{JLX2011}, the complete characterization and enumeration of such self-dual  codes have been given. These results on Euclidean self-dual cyclic codes have been generalized to abelian codes in group algebras  \cite{JLLX2012} and the complete characterization and enumeration of Euclidean self-dual abelian  codes in principal ideal group algebras (PIGAs) have been established.  Extensively, the characterization and enumeration of Hermitian self-dual abelian codes in PIGAs  have been studied in \cite{JLS2013}.  To the best of our knowledge,  the characterization and  enumeration of   self-dual abelian codes in  non-principal ideal group algebras (non-PIGAs) have not been well studied. It is therefore of natural interest to focus on this open problem.
 
 In \cite{JLLX2012} and \cite{JLS2013}, it has been shown that there exists a Euclidean (resp., Hermitian) self-dual abelian  code in $\mathbb{F}_{p^k}[G]$ if and only if $p=2$ and $|G|$ is  even.  In order to  study self-dual abelian codes, it is therefore  restricted to  the group algebra $\mathbb{F}_{2^k}[A\times B]$, where $A$ is an abelian group of odd order and $B$ is a non-trivial  abelian group of two power order. In this case,  $\mathbb{F}_{2^k}[A\times B]$ is a  PIGA  if and only if $B=\mathbb{Z}_{2^{s}}$ is a  cyclic group  (see \cite{FiSe1976}).   Equivalently,   $\mathbb{F}_{2^k}[A\times B]$ is a  non-PIGA   if and only if $B$ is non-cyclic. To avoid tedious computations, we focus on   the simplest case where $B=\mathbb{Z}_2\times \mathbb{Z}_{2^s}$, where $s$ is a positive integer. Precisely, the goal  of this paper is   to determine the  algebraic structures and the  numbers of Euclidean and Hermitian self-dual abelian  codes in $\mathbb{F}_{2^k}[A\times \mathbb{Z}_2\times \mathbb{Z}_{2^s}]$.
 
 It turns out that every  Euclidean (resp., Hermitian) self-dual  abelian code in  $\mathbb{F}_{2^k}[A\times \mathbb{Z}_2\times \mathbb{Z}_{2^s}]$ is a suitable Cartesian product of cyclic codes, Euclidean self-dual cyclic codes, and Hermitian self-dual cyclic codes of length $2^s$ over some Galois extension of the ring $\mathbb{F}_{2^k}+u\mathbb{F}_{2^k}$, where $u^2=0$ (see Section $2$). Hence, the number  of    self-dual abelian  codes in  $\mathbb{F}_{2^k}[A\times \mathbb{Z}_2\times \mathbb{Z}_{2^s}]$ can be determined in terms of  the cyclic  codes mentioned earlier.   Subsequently, useful properties of cyclic codes, Euclidean self-dual cyclic codes, and Hermitian self-dual cyclic codes of length $p^s$ over  $\mathbb{F}_{p^k}+u\mathbb{F}_{p^k}$ are given for all primes $p$. Combining these results,  the characterizations and enumerations of  Euclidean and Hermitian self-dual  abelian codes in $\mathbb{F}_{2^k}[A\times \mathbb{Z}_2\times \mathbb{Z}_{2^s}]$  are rewarded.

 The paper is organized as follows. In Section $2$,  some basic results on abelian codes are recalled together with a link between  abelian codes in $\mathbb{F}_{2^k}[A\times \mathbb{Z}_2\times \mathbb{Z}_{2^s}]$  and cyclic codes of length $2^s$ over  Galois extensions of $\mathbb{F}_{2^k}+u\mathbb{F}_{2^k}$.
 General results on the characterization and enumeration of cyclic codes of  length  $p^s$ over $\mathbb{F}_{p^k}+u\mathbb{F}_{p^k}$ are  provided in Section $3$. In Section $4$,  the   characterizations and enumerations of Euclidean and Hermitian self-dual cyclic codes of length $p^s$ over $\mathbb{F}_{p^k}+u\mathbb{F}_{p^k}$ are established. Summary and remarks are given in Section $5$.
 
 \section{Preliminaries}\label{sec:pre}
 In this section, we  recall some definitions and basic properties of rings and abelian codes.  Subsequently,  a link between an  abelian code  in non-principal ideal algebras and a product of  cyclic codes over rings is given. This link plays an important role in determining  the algebraic structures an the numbers of Euclidean and Hermitian self-dual abelian codes in  non-principal ideal algebras.

 \subsection{Rings and Abelian Codes in Group Rings}
 
 For a prime $p$ and a positive integer $k$,  denote by $\mathbb{F}_{p^k}$ the finite field of order $p^k$.  Let $\mathbb{F}_{p^k}+u\mathbb{F}_{p^k}:=\{a+ub\mid a,b\in \mathbb{F}_{p^k}\}$ be a ring, where the addition and  multiplication are  defined  as in the usual polynomial ring over $\mathbb{F}_{p^k}$ with indeterminate $u$ together with  the condition $u^2=0$. We note that  $\mathbb{F}_{p^k}+u\mathbb{F}_{p^k}$  is isomorphic to $\mathbb{F}_{p^k}[u]/\langle u^2\rangle$ as rings.  
 The  Galois extension of $\mathbb{F}_{p^k}+u\mathbb{F}_{p^k}$  of degree $m$ is defined to be the quotient ring $(\mathbb{F}_{p^k}+u\mathbb{F}_{p^k})[x]/\langle f(x)\rangle$, where $f(x)$ is an irreducible polynomial of degree $m$ over $\mathbb{F}_{p^k}$. It is not difficult to see that  the Galois extension of $\mathbb{F}_{p^k}+u\mathbb{F}_{p^k}$  of degree $m$ is isomorphic to  $\mathbb{F}_{p^{km}}+u\mathbb{F}_{p^{km}}$ as rings.  In the case where  $k$ is even, the  mapping    
 $a+ub\mapsto a^{p^{k/2}}+ub^{p^{k/2}}$
 is a ring automorphism of order $2$ on $\mathbb{F}_{p^k}+u\mathbb{F}_{p^k}$.  
 The readers may refer to \cite{D2010,HQ2009} for properties of the ring $\mathbb{F}_{p^k}+u\mathbb{F}_{p^k}$.

 For a commutative ring $R$ with identity $1$ and a finite abelian group $G$, written additively,  let  $R[G]$ denote the {\it group ring} of $G$ over~$R$. The elements in $ R[G]$ will be written as $\sum\limits_{g\in G}\alpha_{{g }}Y^g $,
 where $ \alpha_{g }\in R$.  The addition and the multiplication in $ R[G]$ are  given as in the usual polynomial ring over $R$ with indeterminate $Y$, where the indices are computed additively in $G$. 
 Note that  $Y^0:=1$ is the multiplicative  identity of   $R[G]$ (resp., $R$), where $0$ is the identity of $G$.   We define  a {\em conjugation} $\bar{~}$  on ${R}[G]$ to be the  map that fixes $R$ and   sends $Y^g$ to $Y^{-g}$ for all $g\in G$, $i.e.$, for $\boldsymbol{u}=\sum\limits_{g\in G}\alpha_{{g }}Y^g \in R[G]$, we set $\overline{\boldsymbol{u}}:= \sum\limits_{g\in G}\alpha_{{g }}Y^{-g}=\sum\limits_{g\in G}\alpha_{{-g }}Y^g  $.  In the case where, there exists a ring automorphism $\rho$ on $R$ of order $2$,  we define $\widetilde{\boldsymbol{u}}:= \sum\limits_{g\in G}\rho(\alpha_{{g }})Y^{-g}$ for all $\boldsymbol{u}=\sum\limits_{g\in G}\alpha_{{g }}Y^g \in R[G]$.
 In the case where $R$ is a finite field $\mathbb{F}_{p^k}$, then $\mathbb{F}_{p^k}[G]$ can be viewed as an $\mathbb{F}_{p^k}$-algebra and  it is called a {\em group algebra}.  The group algebra  $\mathbb{F}_{p^k}[G]$ is called a {\em principal ideal group algebra (PIGA)} if its ideals are generated by a single element.

 An  {\em abelian code} in $R[G]$ is  defined to be an
 ideal in  $R[G]$.  If  $G$ is cyclic, this code is called  a  {\em cyclic code}, a code which is invariant  under the right cyclic shift.    It is well known that cyclic codes of length $n$ over $R$ can be regarded as ideals in the quotient polynomial ring $R[x]/\langle x^n-1\rangle \cong R[\mathbb{Z}_n]$.

 The {\em Euclidean inner product}  in $R[G]$ is defined as follows. For
 \[ \boldsymbol{u}=\sum_{g\in   G} \alpha_{g}Y^{g} \text{ and }\boldsymbol{v}=\sum_{g\in   G} \beta_{g}Y^{g}\]
 in $R[G]$, we set
 \begin{align*}
 \langle \boldsymbol{u},\boldsymbol{v}\rangle_{\rm E}:=\sum_{g\in   G} \alpha_{g}\beta_{g}.
 \end{align*}
 In addition, if  there exists a ring automorphism $\rho$ of order $2$ on $R$,   the $\rho$-{\em inner  product} of $\boldsymbol{u}$ and $\boldsymbol{v}$ is defined by  
 \begin{align*}
 \langle \boldsymbol{u},\boldsymbol{v}\rangle_{\rho}: =\sum_{g\in   G} \alpha_{g}\rho{(\beta_{g})}.
 \end{align*}	
 If $R=\mathbb{F}_{q^2}$ (resp., $R=\mathbb{F}_{q^2}+u\mathbb{F}_{q^2}$) and $\rho(a)=a^q$ (resp., $\rho(a+ub)=a^q+ub^q$) for all $a \in \mathbb{F}_{q^2}$ (resp., $a+ub\in \mathbb{F}_{q^2}+u\mathbb{F}_{q^2}$), the $\rho$-{inner  product}   is called the {\em Hermitian inner product} and denoted by $\langle \boldsymbol{u},{\boldsymbol{v}}\rangle _{\rm H}$.

 The {\em Euclidean dual} and  {\em Hermitian dual}  of $C$ in $R[G]$  are defined to be the sets
 \[C^{\perp_{\rm E}}:= \{\boldsymbol{u}\in R[G]\mid  \langle \boldsymbol{u},\boldsymbol{v}\rangle_{\rm E}=0 \text{ for all } \boldsymbol{v}\in C\}\]  and  \[C^{\perp_{\rm H}}:= \{\boldsymbol{u}\in R[G]\mid  \langle \boldsymbol{u},\boldsymbol{v}\rangle_{\rm H}=0 \text{ for all } \boldsymbol{v}\in C\},\] respectively.

 
 An abelian   code $C$ is said to be  Euclidean self-dual (resp., Hermitian self-dual)   if $C=C^{\perp_{\rm E}}$    (resp., $C=C^{\perp_{\rm H}}$).
 
 For convenience, denote by $N(p^k,n)$,  $NE(p^k,n)$, and $NH(p^k,n)$ the number of cyclic codes,  the number of Euclidean self-dual cyclic codes, and the number of Hermitian  self-dual cyclic codes of length $n$ over $\mathbb{F}_{p^k}+u\mathbb{F}_{p^k}$, respectively.

 \subsection{Decomposition of  Abelian Codes  in $\mathbb{F}_{2^k}[A\times\mathbb{Z}_2\times \mathbb{Z}_{2^s}]$}
 In \cite{JLLX2012} and \cite{JLS2013}, it has been shown that there exists a Euclidean (resp., Hermitian) self-dual abelian code in $\mathbb{F}_{p^k}[G]$ if and only if $p=2$ and $|G|$ is  even.  To  study self-dual abelian codes,  it is sufficient to focus on   $\mathbb{F}_{2^k}[A\times B]$, where $A$  is an abelian group of odd order and $B$ is a non-trivial  abelian group of  two power order. In this case,  $\mathbb{F}_{2^k}[A\times B]$ is a  PIGA  if and only if $B=\mathbb{Z}_{2^{s}}$ is a cyclic group for some positive integer $s$ (see \cite{FiSe1976}).   The complete characterization and enumeration of  self-dual abelian codes in PIGAs have been given in  \cite{JLLX2012} and \cite{JLS2013}. Here, we focus on self-dual  abelian codes in non-PIGAs, or equivalently,    $B$ is non-cyclic.    To avoid  tedious computations, we establish results for the simplest case where $B=\mathbb{Z}_2\times \mathbb{Z}_{2^s}$. Useful decompositions of  $\mathbb{F}_{2^k}[A\times\mathbb{Z}_2\times \mathbb{Z}_{2^s}]$  are  given in this section.

 First, we consider the decomposition of  $\mathcal{R}:=\mathbb{F}_{2^k}[A]$.  In this case, $\mathcal{R}$ is semi-simple \cite{Be1967_2} which can be  decomposed  using the Discrete Fourier Transform  in  \cite{RS1992}   (see \cite{JLS2013} and \cite{JLLX2012} for more details).  For completeness, the decompositions used in this paper are summarized as follows.

 For an  odd positive integer $i$ and a positive integer $k$, let ${\rm ord}_i(2^k)$ denote the multiplicative order of $2^k$ modulo $i$. For each $a\in A$, denote by ${\rm ord}(a)$ the additive order of $a$ in $A$. A {\it $2^k$-cyclotomic class}   of $A$ containing $a\in A$, denoted by $S_{2^k}(a)$, is defined to be the set
 \begin{align*}
 S_{2^k}(a):=&\{2^{ki}\cdot a \mid i=0,1,\dots\}
 =\{2^{ki}\cdot a \mid 0\leq i< {\rm ord}_{{\rm ord}(a)}(2^k) \}, 
 \end{align*}
 where $2^{ki}\cdot a:= \sum\limits_{j=1}^{2^{ki}}a$ in $A$.

 An {\em idempotent} in $\mathcal{R}$ is a non-zero element $e$ such that $e^2=e$. It is called {\em primitive} if for every other idempotent $f$, either $ef=e$ or $ef=0$.  The primitive idempotents in $\mathcal{R}$ are induced by the $2^k$-cyclotomic classes of $A$ (see \cite[Proposition II.4]{DKL2000}).   Let $\{ a_1,a_2,\dots, a_t\}$ be a complete set of representatives of $2^k$-cyclotomic classes  of $A$  and let $e_i$ be the primitive idempotent induced by $S_{2^k}(a_i)$  for all $1\leq i\leq t$.  From \cite{RS1992},   $\mathcal{R}$ can be decomposed as 
 \begin{align}\label{eq-decom0}
 \mathcal{R}= \bigoplus_{i=1}^t\mathcal{R}e_i ,
 \end{align}
 and hence,
 
 \begin{align}\label{eq-decom01}
 \mathbb{F}_{2^k}[A\times\mathbb{Z}_2]   \cong  \mathcal{R}[\mathbb{Z}_2]\cong \bigoplus_{i=1}^t(\mathcal{R}e_i )[\mathbb{Z}_2].
 \end{align}
 It is well known (see \cite{JLLX2012,JLS2013}) that $\mathcal{R}e_i:=\mathbb{F}_{2^{k_i}}$,  where  $k_i$ is a multiple of $k$.  Precisely, $k_i=k|S_{2^k}(a_i)| =k\cdot  {\rm ord}_{{\rm ord}(a_i)}(2^k)$  provided that   $e_i$ is induced by $S_{2^k}(a_i)$.  It follows that $\mathcal{R}e_i[\mathbb{Z}_2]\cong \mathbb{F}_{2^{k_i}}[\mathbb{Z}_2]$.   Under the ring isomorphism that fixes the elements in  $\mathbb{F}_{2^{k_i}}$ and  $Y^1\mapsto u+1$, $\mathbb{F}_{2^{k_i}}[\mathbb{Z}_2]$ is isomorphic to the ring  $\mathbb{F}_{2^{k_i}}+u\mathbb{F}_{2^{k_i}}$, where $u^2=0$.  We note that this ring plays an important role in coding theory and  codes over rings in this family have extensively been studied \cite{D2010,HQ2009,DNS2016,JLU2012}  and references therein.   
 
 From \eqref{eq-decom01} and the  ring isomorphism discussed above,   we have

 \begin{align}\label{eq-decom1}
 \mathbb{F}_{2^k}[A\times \mathbb{Z}_2\times \mathbb{Z}_{2^s}]\cong  \prod_{i=1}^t\mathcal{R}_i[ \mathbb{Z}_{2^s}] ,
 \end{align}
 where $\mathcal{R}_i:=\mathbb{F}_{2^{k_i}}+u\mathbb{F}_{2^{k_i}} $  for all $1\leq i\leq t$.

 In order to study the algebraic structures of  Euclidean and Hermitian self-dual abelian codes in  $\mathbb{F}_{2^k}[A\times \mathbb{Z}_2\times \mathbb{Z}_{2^s}]$, the two rearrangements  of $\mathcal{R}_i$'s in the decomposition \eqref{eq-decom1} are needed. The details are given in the following two subsections.

 \subsubsection{Euclidean Case}

 A $2^k$-cyclotomic class $S_{2^k}(a)$ is said to be of  {\em type} ${\I}$    if  $a=-a$ (in this case, $S_{2^k}(a)=S_{2^k}(-a)$), {\em type} ${\II}$  if $S_{2^k}(a)=S_{2^k}(-a)$ and $a\neq -a$, or {\em type} ${\III}$   if $S_{2^k}(-a)\neq S_{2^k}(a)$. 	The primitive idempotent $e$   {induced by} $S_ {2^k}(a)$   is said to be of type $\lambda\in\{\I,\II,\III\}$ if $S_{2^k}(a)$ is a $2^k$-cyclotomic class of type $\lambda$. 
 
 Without loss of generality, the representatives $a_1, a_2, \dots, a_t$ of  $2^k$-cyclotomic classes   of $A$  can be  chosen such that $\{a_i\mid i=1,2,\dots,{r_{\I}}\}$, $\{a_{r_{\I}+j} \mid j=1,2,\dots,r_{\II}\}$ and $\{a_{r_{\I}+r_{\II}+l}, a_{r_{I}+r_{\II}+r_{\III}+l}=-a_{r_{I}+r_{\II}+l} \mid l=1,2,\dots, r_{\III}\}$ are  sets of representatives of $2^k$-cyclotomic classes of types $\I, \II$, and ${\III}$, respectively, where $t=r_{\I}+r_{\II}+2r_{\III}$.   
 
 Rearranging the terms in the decomposition of  $\mathcal{R}$ in  \eqref{eq-decom0} based on the $3$ types primitive idempotents,  we have 
 \begin{align}
 \mathbb{F}_{2^k}[A\times\mathbb{Z}_2]    &\cong  \bigoplus_{i=1}^t(\mathcal{R}e_i )[\mathbb{Z}_2]  \cong \left( \prod_{i=1}^{r_{\I}}\mathcal{R}_i\right) \times \left( \prod_{j=1}^{r_{\II}} \mathcal{S}_j\right) \times \left( \prod_{l=1}^{r_{\III}} (\mathcal{T}_l\times \mathcal{T}_l)\right), \label{eqSemiSim}
 \end{align}
 where $  \mathcal{R}_i:= \mathbb{F}_{2^{k}}+u\mathbb{F}_{2^{k}}$ for all $i=1,2,\dots, r_{\I}$, $ \mathcal{S}_j:= \mathbb{F}_{2^{k_{r_{\I}+j}}}+u\mathbb{F}_{2^{k_{r_{\I}+j}}}  $ for all   $j=1,2,\dots, r_{\II}$,  and  $ \mathcal{T}_l  := \mathbb{F}_{2^{k_{r_{\I}+r_{\II}+l}}}+u\mathbb{F}_{2^{k_{r_{\I}+r_{\II}+l}}}   $      for all $l=1,2,\dots, r_{\III}$.

 From \eqref{eqSemiSim}, we have  
 \begin{align}\label{eqAbel}
 \mathbb{F}_{2^k}[A\times\mathbb{Z}_2\times\mathbb{Z}_{2^s}]      \cong \left( \prod_{i=1}^{r_{\I}}\mathcal{R}_i[\mathbb{Z}_{2^s}]\right) \times \left( \prod_{j=1}^{r_{\II}} \mathcal{S}_j[\mathbb{Z}_{2^s}]\right) \times \left( \prod_{l=1}^{r_{\III}} (\mathcal{T}_l[\mathbb{Z}_{2^s}]\times \mathcal{T}_l[\mathbb{Z}_{2^s}])\right).
 \end{align} 
 It follows that,  an abelian code $C$ in $ \mathbb{F}_{2^k}[A\times\mathbb{Z}_2\times\mathbb{Z}_{2^s}]  $ can be viewed as 
 \begin{align}\label{decomC} 
 C\cong \left(\prod_{i=1}^{r_{\I}} B_i  \right)\times \left(\prod_{j=1}^{r_{\II}} C_j  \right)\times \left(\prod_{l=1}^{r_{\III}} \left( D_l\times D_l^\prime\right) \right), \end{align}
 where $B_i$, $C_j$, $D_s$ and $D_s^\prime$ are   cyclic   codes in  $\mathcal{R}_i[\mathbb{Z}_{2^s}]$,       $\mathcal{S}_j[\mathbb{Z}_{2^s}]$, $\mathcal{T}_l[\mathbb{Z}_{2^s}]$ and $\mathcal{T}_l[\mathbb{Z}_{2^s}]$, respectively,  for all  $i=1,2,\dots,r_{\I}$, $j=1,2,\dots,r_{\II}$ and  $l=1,2,\dots,r_{\III}$.

 Using the analysis similar to those in   \cite[Section II.D]{JLLX2012}, the Euclidean dual of $C$  in (\ref{decomC}) is of the 
 form 
 
 \begin{align*} 
 C^{\perp_{\rm E}}\cong \left(\prod_{i=1}^{r_{\I}} B_i^{\perp_{\rm E}}  \right)&\times \left(\prod_{j=1}^{r_{\II}} C_j ^{\perp_{\rm H}} \right)\times \left(\prod_{l=1}^{r_{\III}} \left( (D_l^\prime) ^{\perp_{\rm E}}\times  D_l^{\perp_{\rm E}}\right) \right).
 \end{align*}
 Similar to {\cite[Corollary 2.9]{JLLX2012}}, necessary and sufficient conditions for an abelian code in $\mathbb{F}_{2^k}[A\times\mathbb{Z}_2\times\mathbb{Z}_{2^s}] $ to be Euclidean self-dual  can be  given using  the notions of cyclic codes of length $2^s$   over $\mathcal{R}_i$, $\mathcal{S}_j$,  and $\mathcal{T}_l$ in  the following corollary.
 
 \begin{corollary}\label{selfDuA}
     An abelian code $C$ in    $\mathbb{F}_{2^k}[A\times\mathbb{Z}_2\times\mathbb{Z}_{2^s}] $ is  Euclidean self-dual if and only if in the decomposition (\ref{decomC}), 
     \begin{enumerate}[$i)$]
         \item $B_i$ is   a Euclidean self-dual  cyclic code of length $2^s$ over $\mathcal{R}_i$ for all  all $i=1,2,\dots,r_{\I}$,
         \item $C_j$ is  a Hermitian  self-dual  cyclic code of length $2^s$ over $\mathcal{S}_j$  for all $j=1,2,\dots,r_{\II}$, and       
         \item   $D_l^\prime= D_l^{\perp_{\rm E}}$ is a  cyclic code of length $2^s$ over $\mathcal{T}_l$  for all $l=1,2,\dots,r_{\III}$.
        \end{enumerate}
    \end{corollary}
    
    Given a positive integer   $k$   and an odd  positive integer $j$, the pair $(j,2^k)$  is said to be {\em  good} if $j$ divides $2^{kt}+1$ for some  integer $t\geq 1$,   and {\em bad} otherwise.  This  notions have been introduced in \cite{JLX2011,JLLX2012} for  the enumeration of  self-dual cyclic codes  and self-dual abelian codes over finite fields.

    Let  $\chi$  be a function  defined on the pair $(j,2^k)$, where $j$ is an odd  positive integer, as follows.
    \begin{align}\label{chi}
    \chi(j,2^k)
    =\begin{cases}
    0 &\text{ if } (j,2^k) \text{ is good},\\
    1 &\text{ otherwise.}
    \end{cases}
    \end{align}
    The number of Euclidean self-dual abelian codes in  $\mathbb{F}_{2^k}[A\times\mathbb{Z}_2\times\mathbb{Z}_{2^s}] $ can be determined as follows.
    
    \begin{theorem} \label{NEA}
        Let   $k$ and $s$  be positive  integers   and  let $A$ be a finite abelian group of odd order and exponent $M$.
        Then the number of Euclidean self-dual abelian codes in  $\mathbb{F}_{2^k}[A\times\mathbb{Z}_2\times\mathbb{Z}_{2^s}] $ is 
        \begin{align*} 
        \left(NE(2^k,2^s)\right)^{ \sum\limits_{{d\mid M, {\rm ord}_d(2^k)= 1}} (1-\chi(d,2^k))\mathcal{N}_A(d)}  &\times\prod_{\substack{d\mid M\\ {\rm ord}_d(2^k)\ne 1}} \left(NH(2^{k\cdot {\rm ord}_d(2^k)},2^s)\right)^{(1-\chi(d,2^k))\frac{\mathcal{N}_A(d)}{{\rm ord}_d(2^k)}}  \\
        &\times\prod_{d\mid M} \left(N( 2^{k\cdot {\rm ord}_d(2^k)},2^s)\right)^{\chi(d,2^k)\frac{\mathcal{N}_A(d)}{2{\rm ord}_d(2^k)}} ,
        \end{align*}
        where  $\mathcal{N}_A(d)$ denotes the number of elements in $A$   of order $d$ determined in \cite{B1997}.
    \end{theorem}
    \begin{proof} 
        From (\ref{decomC}) and Corollary~\ref{selfDuA}, it suffices  to determine the numbers of   cyclic codes $B_i$'s,   $C_j$'s, and   $D_l$'s such that $B_i$ and   $C_j$  are Euclidean and Hermitian self-dual,  respectively.
        
        From \cite[Remark 2.5]{JLS2013}, the elements in $A$ of the same order are partitioned into $2^k$-cyclotomic classes of the same type.
        For each divisor $d$ of $M$, a  $2^k$-cyclotomic class containing an element of order $d$ has cardinality 
        ${{\rm ord}_d(2^k)}$  and the number of such  $2^k$-cyclotomic classes   is $\frac{\mathcal{N}_A(d)}{{\rm ord}_d(2^k)}$. We consider the following $3$ cases.
        
        \noindent {\bf  Case 1:}   $\chi(d,2^k)=0$  and  ${\rm ord}_d(2^k)=1$.   By \cite[Remark 2.6]{JLLX2012}, every  $2^k$-cyclotomic class    of $A$  containing an element  of order $d$  is of type  $\I$.  Since there are  $\frac{\mathcal{N}_A(d)}{{\rm ord}_d(2^k)}$ such $2^k$-cyclotomic classes,  the  number  of Euclidean self-dual cyclic  codes $B_i$'s of length $2^s$ corresponding to $d$ is 
        \[\left(NE( 2^k,2^s)\right)^{ \frac{\mathcal{N}_A(d)}{{\rm ord}_d(2^k)}}=\left(NE( 2^k,2^s)\right)^{   (1-\chi(d,2^k))\mathcal{N}_A(d)}  .\]

        \noindent {\bf  Case 2:}   $\chi(d,2^k)=0$ and  ${\rm ord}_d(2^k)\ne 1$.
        By \cite[Remark 2.6]{JLLX2012}, every  $2^k$-cyclotomic class    of $A$  containing an element   of order $d$  is of type  $\II$.  Hence,  the  number  of Hermitian self-dual cyclic  codes $C_j$'s of length $2^s$  corresponding to $d$ is 
        \begin{align*} \left(NH( 2^{k\cdot {\rm ord}_d(2^k)} ,2^s)\right)^{ \frac{\mathcal{N}_A(d)}{{\rm ord}_d(2^k)}} =\left(NH( 2^{k\cdot {\rm ord}_d(2^k)} ,2^s)\right)^{(1-\chi(d,2^k))\frac{\mathcal{N}_A(d)}{{\rm ord}_d(2^k)}}.
        \end{align*}
        
        \noindent {\bf  Case 3:}   $\chi(d,2^k)=1$.  By \cite[Lemma 4.5]{JLLX2012},  every  $2^k$-cyclotomic class    of $A$  containing an element   of order $d$  is of type  $\III$.  Then  the  number  of   cyclic  codes $D_l$'s of length $2^s$  corresponding to $d$ is 
        \[\left(N( 2^{k\cdot{\rm ord}_d(2^k)},2^s)\right)^{ \frac{\mathcal{N}_A(d)}{2{\rm ord}_d(2^k)}}=\left(N( 2^{k\cdot{\rm ord}_d(2^k)},2^s)\right)^{\chi(d,2^k)\frac{\mathcal{N}_A(d)}{2{\rm ord}_d(2^k)}}.\]

        The desired result follows since $d$ runs over all divisors of $M$.
    \end{proof}

    This enumeration will be completed by counting   
    the above  numbers  $NE$, $NH$, and $N$ in  Corollaries \ref{countesdcycliccodes},  \ref{counthsdcycliccodes},  and \ref{cor-cyclic-summ},  respectively.

    \subsubsection{Hermitian Case}
    
    We focus on the case where $k$ is even.
    A $2^k$-cyclotomic class $S_{2^k}(a)$ is said to be of  {\em type} $\Ip$    if    $S_{2^k}(a)=S_{2^k}(-2^{\frac{k}{2}}a)$ or {\em type} $\IIp$  if $S_{2^k}(a)\ne S_{2^k}(-2^{\frac{k}{2}} a)$. 	The primitive idempotent $e$   {induced by} $S_ {2^k}(a)$   is said to be of type $\lambda\in\{\Ip,\IIp\}$ if $S_{2^k}(a)$ is a $2^k$-cyclotomic class of type $\lambda$. 
    
    Without loss of generality,  the representatives $a_1, a_2, \dots, a_t$ of  $2^k$-cyclotomic classes   can be  chosen such that $\{a_i\mid i=1,2,\dots,{r_\Ip}\}$  and $\{a_{r_\Ip+j}, a_{r_{\Ip}+r_\IIp+j}=-2^\frac{k}{2}a_{r_{\Ip}+j} \mid j=1,2,\dots, r_{\IIp}\}$ are  sets of representatives of $2^k$-cyclotomic classes of types $\Ip$  and $\IIp$, respectively, where $t=r_\Ip+2r_{\IIp}$.   
    
    Rearranging the terms in the decomposition of  $\mathcal{R}$ in  \eqref{eq-decom0} based on the above $2$ types primitive idempotents,  we have 
    \begin{align}
    \mathbb{F}_{2^k}[A\times\mathbb{Z}_2]    &\cong  \bigoplus_{i=1}^t(\mathcal{R}e_i )[\mathbb{Z}_2]  \cong  \left( \prod_{j=1}^{r_{\Ip}} \mathcal{S}_j\right) \times \left( \prod_{l=1}^{r_{\IIp}} (\mathcal{T}_l\times \mathcal{T}_l)\right), \label{eqSemiSim2}
    \end{align}
    where  $ \mathcal{S}_j:= \mathbb{F}_{2^{k_{j}}}+u\mathbb{F}_{2^{k_{j}}}  $ for all   $j=1,2,\dots, r_{\Ip}$  and  $ \mathcal{T}_l  :=    \mathbb{F}_{2^{k_{r_{\Ip}+l}}}+u\mathbb{F}_{2^{k_{r_{\Ip}+l}}}   $      for all $l=1,2,\dots, r_{\IIp}$.

    From \eqref{eqSemiSim2}, we have  
    \begin{align}\label{eqAbel2}
    \mathbb{F}_{2^k}[A\times\mathbb{Z}_2\times\mathbb{Z}_{2^s}]      \cong   \left( \prod_{j=1}^{r_{\Ip}} \mathcal{S}_j[\mathbb{Z}_{2^s}]\right) \times \left( \prod_{l=1}^{r_{\IIp}} (\mathcal{T}_l[\mathbb{Z}_{2^s}]\times \mathcal{T}_l[\mathbb{Z}_{2^s}])\right).
    \end{align} 
    Hence,  an abelian code $C$ in $\mathbb{F}_{2^k}[A\times\mathbb{Z}_2\times\mathbb{Z}_{2^s}] $ can be viewed as 
    \begin{align}\label{decomC2} 
    C\cong   \left(\prod_{j=1}^{r_{\Ip}} C_j  \right)\times \left(\prod_{l=1}^{r_{\IIp}} \left( D_l\times D_l^\prime\right) \right), \end{align}
    where $C_j$, $D_l$ and $D_l^\prime$ are   cyclic   codes in    $\mathcal{S}_j[\mathbb{Z}_{2^s}]$, $\mathcal{T}_l[\mathbb{Z}_{2^s}]$ and $\mathcal{T}_l[\mathbb{Z}_{2^s}]$, respectively,  for all   $j=1,2,\dots,r_{\Ip}$ and  $l=1,2,\dots,r_{\IIp}$.

    Using the analysis similar to those in  \cite[Section II.D]{JLS2013}, the Hermitian  dual of $C$  in (\ref{decomC2}) is of the 
    form

    \begin{align*} 
    C^{\perp_{\rm H}}\cong   \left(\prod_{j=1}^{r_{\Ip}} C_j ^{\perp_{\rm H}} \right)\times \left(\prod_{l=1}^{r_{\IIp}} \left( (D_l^\prime) ^{\perp_{\rm E}}\times  D_l^{\perp_{\rm E}}\right) \right).
    \end{align*}

    Similar to {\cite[Corollary 2.8]{JLS2013}}, necessary and sufficient conditions for an abelian code in $\mathbb{F}_{2^k}[A\times\mathbb{Z}_2\times\mathbb{Z}_{2^s}] $ to be Hermitian self-dual are now given using  the notions of cyclic codes of length $2^s$   over $\mathcal{S}_j$   and $\mathcal{T}_l$ in  the following corollary.
    
    \begin{corollary}\label{selfDuA2}
        An abelian code $C$  in $\mathbb{F}_{2^k}[A\times\mathbb{Z}_2\times\mathbb{Z}_{2^s}] $  is  Hermitian  self-dual if and only if in the decomposition (\ref{decomC2}), 
        \begin{enumerate}[$i)$]
            \item $C_j$ is a  Hermitian   self-dual      cyclic code of length $2^s$ over $\mathcal{S}_j$ for all  $j=1,2,\dots,r_{\Ip}$, and       
            \item   $D_l^\prime= D_l^{\perp_{\rm E}}$  is a  cyclic code of length $2^s$ over $\mathcal{T}_l$   for all $l=1,2,\dots,r_{\IIp}$.
        \end{enumerate}
    \end{corollary}
    
    Given   a positive integer $k$ and  an odd  positive integer $j$, the pair $(j,2^k)$  is said to be {\em oddly good} if $j$ divides $2^{kt}+1$ for some odd  integer $t\geq 1$, and {\em evenly good} if $j$ divides $2^{kt}+1$ for some even  integer $t\geq 2$. These notions have been introduced in  \cite{JLS2013}  for characterizing the Hermitian self-dual abelian codes in PIGAs.    
    
    Let   $\lambda$ be a function defined on the pair $(j,2^k)$, where $j$ is an odd  positive integer, as  
    \begin{align}\label{lambda}
    \lambda(j,2^k)=
    \begin{cases}
    0&\text{if }  (j,2^k) \text{ is oddly good},\\
    1&\text{otherwise}.
    \end{cases}
    \end{align}
    The number of Hermitian self-dual abelian codes in  $\mathbb{F}_{2^k}[A\times\mathbb{Z}_2\times\mathbb{Z}_{2^s}] $ can be determined as follows.
    
    \begin{theorem} \label{NHA}
        Let $k$ be an even positive integer and let $s$ be a positive integer. Let $A$ be a finite abelian group of odd order  and  exponent $M$. Then the number  of  Hermitian self-dual abelian codes in   $\mathbb{F}_{2^k}[A\times\mathbb{Z}_2\times\mathbb{Z}_{2^s}]$ is 
        \begin{align*} 
        \prod_{{d\mid M}} \left(NH( 2^{k\cdot{\rm ord}_d(2^k)}, 2^s)\right)^{(1-\lambda(d,2^\frac{k}{2}))\frac{\mathcal{N}_A(d)}{{\rm ord}_d(2^k)}}  \times\prod_{d\mid M} \left(N( 2^{k\cdot{\rm ord}_d(2^k)},2^s)\right)^{\lambda(d,2^\frac{k}{2})\frac{\mathcal{N}_A(d)}{2{\rm ord}_d(2^k)}}, 
        \end{align*}
        where  $\mathcal{N}_A(d)$ denotes the number of elements of order $d$ in $A$ determined in \cite{B1997}.
    \end{theorem}
    \begin{proof}
        By Corollary~\ref{selfDuA2} and (\ref{decomC2}), it is enough to determine the numbers   cyclic  codes $C_j$'s    and   $D_l$'s of length $2^s$  in (\ref{decomC2})  such that $C_j$ is Hermitian self-dual.    
        
        The desired result can be obtained using arguments  similar to those in the proof of  Theorem \ref{NEA}, where \cite[Lemma 3.5]{JLS2013} is applied instead of  \cite[Lemma 4.5]{JLLX2012}.
    \end{proof}
    
    This enumeration will be completed by counting   
    the above  numbers  $NH$ and $N$ in  Corollaries \ref{counthsdcycliccodes}  and \ref{cor-cyclic-summ},  respectively.

    \section{Cyclic Codes of Length $p^s$ over $\mathbb{F}_{p^k}+u\mathbb{F}_{p^k}$}
    
    The  enumeration of self-dual abelian codes in non-PIGAs in  the previous section requires properties of cyclic codes of length $2^s$ over $\mathbb{F}_{2^k}+u\mathbb{F}_{2^k}$.  In this section, a more general situation is discussed. Precisely,  properties  cyclic of length $p^s$ over   $\mathbb{F}_{p^k}+u\mathbb{F}_{p^k}$ are studied for all primes $p$. We note that algebraic structures of cyclic codes of length $p^s$ over $\mathbb{F}_{p^k}+u\mathbb{F}_{p^k}$  was studied in \cite{HQ2009} and \cite{D2010}. Here, based on \cite{HMK2008}, we give an  alternative characterization of such codes which is useful in studying self-dual cyclic codes of length $p^s$ over $\mathbb{F}_{p^k}+u\mathbb{F}_{p^k}$. 
    
    First,  we note that  there exists a one-to-one correspondence between the cyclic codes of length $p^s$ over~$\mathbb{F}_{p^k}+u\mathbb{F}_{p^k}$ and the  ideals in the quotient ring $(\mathbb{F}_{p^k}+u\mathbb{F}_{p^k})[x]/\langle x^{p^s}-1\rangle $. Precisely,  a  cyclic code $C$ of length $p^s$ can be  represented by the ideal 
    \[\left\{ \left. \displaystyle\sum_{i=0}^{p^s-1} v_ix^i \ \right\vert \ (v_0,v_1,\ldots ,v_{p^s-1})\in C\right\}\]
    in $(\mathbb{F}_{p^k}+u\mathbb{F}_{p^k})[x]/\langle x^{p^s}-1\rangle $.
    
    Form now on, a cyclic code $C$ will be referred to  as the above  polynomial presentation.  Note that the map   $\mu : (\mathbb{F}_{p^k}+u\mathbb{F}_{p^k})[x] /\langle x^{n}-1\rangle \rightarrow \mathbb{F}_{p^k}[x] /\langle x^{n}-1\rangle $   defined by
    \[\mu \left(f(x)\right)= f(x)~(\text{mod~}u) \]
    is a surjective ring homomorphism.  For each  cyclic code   $C$ in   $(\mathbb{F}_{p^k}+u\mathbb{F}_{p^k})[x] /\langle x^{p^s}-1\rangle $ and $i\in\{0,1\}$,  let 
    \[\Tor_i(C)= \{\mu (v(x)) \mid  v(x) \in (\mathbb{F}_{p^k}+u\mathbb{F}_{p^k})[x] /\langle x^{n}-1\rangle \text{ and }    {u}^iv(x) \in C \}.\]
    For each $i \in\{0,1\}$, $\Tor_i(C)$ is called the $i$th \textit{torsion code} of $C$. 
    The codes $\Tor_0({C})=\mu (C)$ and $\Tor_1(C)$ are some time called the \textit{residue} and \textit{torsion codes} of $C$, respectively.  
    
    It is not difficult to see  that for each $i\in \{0,1\}$, $ c(x) \in \Tor_i(C)$ if and only if $u^i(c(x)+uz(x))\in C$ for some ${z}(x)\in \mathbb{F}_{p^k}[x]/\langle x^{p^s}-1\rangle $. Consequently,  we have  that $\Tor_0(C)\subseteq \Tor_1(C)$  are ideals   in $\mathbb{F}_{p^k}[x]/\langle x^{p^s}-1\rangle$ (cyclic codes of length $p^s$ over $\mathbb{F}_{p^k}$).  We note that every ideal  $C$ in  $\mathbb{F}_{p^k}[x]/\langle x^{p^s}-1\rangle$ is of the  form $\langle (x-1)^i\rangle $ for some $0\leq i\leq p^s$ and  the cardinality of $C$ is  $p^{s-i}$. 
    
    From the structures of cyclic codes of length $p^s$ over $\mathbb{F}_{p^k} $ discussed above  and \cite[Proposition 2.5]{D2010}, we have the following properties of the torsion and residue codes.
    \begin{proposition}\label{counttorandc}
        Let $C$ be an ideal in ~$(\mathbb{F}_{p^k}+u\mathbb{F}_{p^k})[x]/\langle x^{p^s}-1\rangle $ and let  $i\in \{0,1\}$. Then the following statements hold.
        \begin{enumerate}[$(i)$]
            \item $\Tor_i(C)$ is an ideal of~$\mathbb{F}_{p^k}[x]/\langle x^{p^s}-1\rangle $ and $\Tor_i(C)=\langle (x-1)^{T_i} \rangle$ for some $0\leq T_i \leq p^s$. 
            \item If $\Tor_i(C)=\langle (x-1)^{T_i} \rangle$, then $|\Tor_i (C) |= (p^k)^{p^s-T_i}$.
            \item $|C|=|\Tor_0 (C) |\cdot|\Tor_1 (C) |=(p^k)^{2p^s-(T_0+T_1)}$.
        \end{enumerate}
    \end{proposition}
    With the notations given in Proposition~\ref{counttorandc},  for  each $i\in \{0,1\}$,   $T_i(C):=T_i$   is called the $i$th-\textit{torsional degree} of~$C$. 
    {
        \begin{remark} From Proposition \ref{counttorandc} and the definition above, we have the following facts.
            \begin{enumerate}[$(i)$]
                \item  Since  $\Tor_0(C)\subseteq \Tor_1(C)$, we have $0\leq T_1(C)\leq T_0(C)\leq p^s$.
                \item  If  $u(x-1)^t \in C$, then $t\geq T_1(C)$.
            \end{enumerate}
        \end{remark}}
        
        Next, we determine a generator set of an ideal in  $(\mathbb{F}_{p^k}+u\mathbb{F}_{p^k})[x]/ \langle x^{p^s}-1 \rangle $.  
        \begin{theorem}\label{formidealsnotunique}
            Let $C$ be an ideal of $(\mathbb{F}_{p^k}+u\mathbb{F}_{p^k})[x]/\langle x^{p^s}-1\rangle $. Then
            \[C=\langle s_{0}(x),us_{1}(x)\rangle ,\]
            where, for each $i\in \{0,1\}$, 
            \begin{enumerate}[$(i)$]
                \item either $s_j(x)=0$ or $s_j(x)=(x-1)^{t_j}+uz_j(x)$ for some $z_j(x) \in \mathbb{F}_{p^k}[x]/\langle x^{p^s}-1\rangle $ and $0\leq t_j<p^s$,
                
                \item $s_j(x)\neq 0$ if and only if $\Tor_j (C)\neq \{0\}$ and $\Tor_0(C) \neq \Tor_1(C)$, and

                \item if $s_j(x)\neq 0$, then $\Tor_j (C)=\langle (x-1)^{t_j}\rangle $.
            \end{enumerate}
        \end{theorem}
        \begin{proof} 

            The statement  can be obtained using a  slight modification of the proof of  \cite[Theorem 6.5]{DP2007}. For completeness,  the details are given as follows. 
            
            For each ideal   $I$  in  $\mathbb{F}_{p^k}[x]/\langle x^{p^s}-1 \rangle $, it  can be represented as $I=\{0\}  $ or $C=\langle (x-1)^{t_0}\rangle $ where $0\leq t_0< p^s$. If $I=\{0\}$, then we are done by choosing  $s_0(x)=0$. For  $0\leq t_0 <p^s$, let $s_0(x)=(x-1)^{t_0}$. By abuse of notation,   $\Tor_0 (C)=C=\langle (x-1)^{t_0}\rangle $. Hence, $s_0(x)$ satisfies the conditions $(i)$, $(ii)$ and $(iii)$.

            Since  $C$ is  an ideal of the ring $(\mathbb{F}_{p^k}+u\mathbb{F}_{p^k})[x]/ \langle x^{p^s}-1 \rangle$ and $\mu$ is a surjective ring homomorphism, $\mu (C)$ is an ideal of $\mathbb{F}_{p^k}[x]/ \langle x^{p^s}-1 \rangle$ which implies that $\mu (C)=\langle s_0^{\prime}(x) \rangle$ where $s_0^{\prime}(x)$ satisfies the conditions $(i)$, $(ii)$ and $(iii)$. If $s_0^{\prime}(x)=0$, then take $s_0(x)=0$. Assume that $s_0^{\prime}(x)\neq 0$, then $s_0^{\prime}(x)=(x-1)^{t_0}$ where $0\leq t_0 <p^s$. Then there exists $z_0(x)\in \mathbb{F}_{p^k}[x]/ \langle x^{p^s}-1 \rangle$ such that $s_0(x)=(x-1)^{t_0}+uz_0(x)\in C$ and $\mu (s_0(x))=s_0^{\prime}(x)$, i.e.,  $s_0(x)=s_0^{\prime}(x) +uz_0(x)$. Since $\Tor_1(C)$ is an ideal of $\mathbb{F}_{p^k}[x]/ \langle x^{p^s}-1 \rangle $, it follows that $\Tor_1(C)=\langle (x-1)^{t_1}\rangle $ for some $0\leq t_1\leq p^s$. Let $s_1(x)=(x-1)^{t_1}$. Claim that $C=\langle s_0(x),us_1(x)\rangle $. Since $C$ is an ideal of $(\mathbb{F}_{p^k}+u\mathbb{F}_{p^k})[x]/\langle x^{p^s}-1\rangle $, we have  $u(x-1)^{t_1}\in C$. Thus, $\langle s_0(x),us_1(x)\rangle \subseteq C$. To show that $C\subseteq \langle s_0(x),us_1(x)\rangle $, let $c(x)\in C$. Then $\mu (c(x))=w(x)s_0^{\prime}(x)$ for some $w (x)\in \mathbb{F}_{p^k}[x]/ \langle x^{p^s}-1 \rangle$. Thus, $\mu (c(x))=w(x)s_0^{\prime}(x)$ which implies that $c(x)=w(x)s_0^{\prime}(x)+uw^{\prime}(x)=w(x)s_0(x)-uw(x)z_0(x)+uw^{\prime}(x)=w(x)s_0(x)+u\left(-z_0(x)w(x)+w^{\prime}(x)\right)$ for some $w^{\prime}(x) \in \Tor_1(C)$. Since $w^{\prime}(x)\in \Tor_1(C)$, it follows that $c(x)\in \langle s_0(x),us_1(x)\rangle $. Therefore, $C=\langle s_0(x),us_1(x)\rangle $ as desired.

            Note that $s_1(x)=(x-1)^{p^k}=0$ implies $C=\{0\}$. Assume that $C\neq \{0\}$. Then   $s_1(x)=(x-1)^{t_1}$ where $0\leq t_1<p^s$. If $s_0(x)=0$, then we are done. Assume that $s_0(x)\neq 0$. Then $s_0(x)=(x-1)^t$ where $0\leq t< p^s$. Hence,  $\Tor_0(C)=\langle (x-1)^t\rangle $. Since $\Tor_0(C)\subseteq \Tor_1(C)$, we have $t_1\leq t$. If $t_1<t$, then we are done. Assume that $t_1=t$. Then $s_0(x)=(x-1)^t+uz_0(x)$ for some $z_0(x) \in \mathbb{F}_{p^k}[x]/ \langle x^{p^s}-1 \rangle $. It follows that  $us_1(x)=u(x-1)^{t_1}=u(x-1)^t=us_0(x)\in \langle s_0(x)\rangle $, a contradiction. Therefore, $C=\langle s_0(x)\rangle =\langle s_0(x),us_0(x)\rangle =\langle s_0(x),us_1(x)\rangle $.
        \end{proof}

        However, the generator set given in Theorem \ref{formidealsnotunique}  does not need to be  unique. The unique presentation is given in the following theorem.
        \begin{theorem}\label{formidealunique}
            Let $C$ be an ideal  in  $(\mathbb{F}_{p^k}+u\mathbb{F}_{p^k})[x]/\langle x^{p^s}-1\rangle$,  $T_0:=T_0(C)$ and $T_1:=T_1(C)$. Then
            \[C=\langle f_{0}(x),f_{1}(x)\rangle ,\]
            where 
            \[    f_{0}(x)= \begin{cases}
            (x-1)^{T_0}+u(x-1)^th(x)  & \text{ if~~~ }T_{0}<p^s,\\
            0 & \text{ if~~~ }T_0=p^s,
            \end{cases} \]
            and
            
            \[    f_{1}(x)= \begin{cases}
            u(x-1)^{T_1}  & \text{ if~~~ }T_{1}<p^s,\\
            0 & \text{ if~~~ }T_1=p^s,
            \end{cases}    \]
            with $h(x)\in \mathbb{F}_{p^k}[x]/\langle x^{p^s}-1\rangle $ is either zero or a unit  with $t+ \deg(h(x))<T_0$.
            
            Moreover, $(f_0(x),f_1(x))$ is unique in the sense  that  if there exists  a  pair $(g_0(x),g_1(x)) $ of polynomials satisfying the conditions in the theorem, then $f_0(x)=g_0(x)$ and $f_1(x)=g_1(x)$.
        \end{theorem}
        
        \begin{proof}
            If  $C=\{0\}$, then  $\Tor _1(C)=\{0\}$ and $\Tor_0 (C)=\{0\}$ which imply  that $T_0=p^s$ and $T_{1}=p^s$. The polynomials $f_0(x)=0$ and $f_1(x)=0$ have the desired properties.

            Next, assume that $C\neq \{0\}$. Then there exists the smallest nonnegative integer~$r\in\{0,1\}$ such that $T_r<p^s$.  From  Theorem \ref{formidealsnotunique}, it can be concluded that  \[C=\langle s_{0}(x),us_{1}(x)\rangle ,\]
            where 
            
            \[    s_{0}(x)= \begin{cases}
            (x-1)^{T_0}+ug(x) \text{ for some }g(x)\in \mathbb{F}_{p^k}[x]/\langle x^{p^s}-1\rangle   & \text{ if } r=0,\\
            0& \text{ if  } r=1 ,
            \end{cases}    \]
            and
            \[ s_{1}(x)= (x-1)^{T_1}.\]
            
            \noindent \textbf{Case~1:} $r=0$.
            Then $s_0(x)=(x-1)^{T_0}+ug(x) \text{ for some }g(x)\in \mathbb{F}_{p^k}[x]/\langle x^{p^s}-1\rangle $ and $s_{1}(x)= (x-1)^{T_1}$. It follows that  $C=\langle (x-1)^{T_0}+ug(x), (x-1)^{T_1} \rangle $, $T_1(C)=T_1$ and $T_0(C)=T_0$. Let $f_1(x)=us_1(x)$. 
            Since $g(x) \in \mathbb{F}_{p^k}[x]/\langle x^{p^s}-1\rangle $, we have 
            \[s_0(x)=(x-1)^{T_0}+u \displaystyle\sum_{j=0}^{p^s-1}a_j(x-1)^j,\]
            where $a_j \in \mathbb{F}_{p^k}$ for all $j=0,1,\ldots , p^k-1$. Since $u(x-1)^{T_1}=us_1(x)\in C$,  we have  $ua_j(x-1)^j \in C$ for all $j=T_1,T_1+1, \ldots , p^s-1$.  It follows that  $u\displaystyle\sum_{j=T_1}^{p^s-1}a_j(x-1)^j\in C$.
            Let $f_0(x)=(x-1)^{T_0}+u\displaystyle\sum_{j=0}^{T_1-1}a_j(x-1)^j$. Then $f_0(x)=s_0(x)-u\displaystyle\sum_{j=T_1}^{p^s-1}a_j(x-1)^j\in C$.

            We show that $C =\langle f_0(x),f_1(x)\rangle $.  From the discussion above, we have  $\langle f_0(x),f_1(x)\rangle \subseteq C $. Since $us_1(x)=u(x-1)^{T_1}=f_1(x)$, it follows that $ua_j(x-1)^j \in \langle f_0(x),f_1(x)\rangle $ for all $j=T_1,T_1+1, \ldots ,p^s-1$.  Hence,  $u\displaystyle\sum_{j=T_1}^{p^s-1}a_j(x-1)^j\in \langle f_0(x),f_1(x)\rangle $ which implies that \[s_0 (x) =f_0(x)+ u\displaystyle\sum_{j=T_1}^{p^s-1}a_j(x-1)^j \in \langle f_0(x),f_1(x)\rangle .\] Therefore, $C=\langle s_0(x),us_1(x)\rangle \subseteq \langle f_0(x),f_1(x)\rangle $. As desired, $C=\langle f_0(x),f_1(x)\rangle $.
            
            \noindent 	\textbf{Case~2:} $r=1$. Then $s_0(x)=0$ and $s_1(x)=u(x-1)^{T_1}$ which implies that $\Tor_1(C)=\langle (x-1)^{T_1}\rangle $ and $\Tor_0(C)=\{0\}$. By choosing  $f_0(x)=0$ and $f_1(x)=u(x-1)^{T_1}$,  the result follows.

            To prove the uniqueness, let $C=\langle g_0(x),g_1(x)\rangle $ be such that $g_0(x)$ and $g_1(x)$ satisfying the conditions in the theorem. Then $g_1(x)=u(x-1)^{T_1}=f_1(x)$.
            
            Write $g_0(x)=(x-1)^{T_0}+u\displaystyle\sum_{j=0}^{T_1-1}c_j(x-1)^j$ where $c_j\in \mathbb{F}_{p^k}$. Then \[f_0(x)-g_0(x)=u\displaystyle\sum_{j=0}^{T_1-1}(a_j-c_j)(x-1)^j.\]  It can be seen  that $f_0(x)-g_0(x)=u(x-1)^{l} h(x)$,  where $h(x)=0$ or $h(x)$ is a unit with $l \leq T_1-1 <T_1$. If $h(x)$ is a unit, then  $u(x-1)^{l} \in C$ which implies that  $l\geq T_1$, a contradiction. Hence, $h(x)=0$   which means that  $f_0(x)=g_0(x)$ as desired.
        \end{proof}
        \begin{definition}
            For each ideal  $C$ in $(\mathbb{F}_{p^k}+u\mathbb{F}_{p^k})[x] / \langle x^{p^s}-1\rangle $, denote by  $C=\langle \langle f_0(x), f_1(x)\rangle \rangle $   the unique representation of the ideal~$C$ obtained in Theorem~\ref{formidealunique}.
        \end{definition}
        Illustrative examples of the representations   in Theorem~\ref{formidealsnotunique} and Theorem~\ref{formidealunique} are given as follows.
        \begin{example}
            Consider the ideal $C=\langle (x-1)^2\rangle$ in   $(\mathbb{F}_2+u\mathbb{F}_2 )[x] / \langle x^4-1 \rangle$. Using Theorem~\ref{formidealunique}, we obtain that $C$ has the unique representation $\langle \langle (x-1)^{2},u(x-1)^{2} \rangle \rangle $.   Based on  Theorem~\ref{formidealsnotunique}, $C$ can be  represented as  $\langle (x-1)^2,0\rangle $, $\langle (x-1)^2+u(x-1)^2,0\rangle $,  and $\langle (x-1)^2+u(x-1)^3,0\rangle $.
        \end{example}
        
        The annihilator of an ideal $C$  in  $(\mathbb{F}_{p^k}+u\mathbb{F}_{p^k})[x] / \langle x^{p^s}-1\rangle $ is key to determine properties $C$ as well as the number of ideals in $(\mathbb{F}_{p^k}+u\mathbb{F}_{p^k})[x] / \langle x^{p^s}-1\rangle $. 
        \begin{definition}\label{defann}
            Let $C$ be an ideal in  $(\mathbb{F}_{p^k}+u\mathbb{F}_{p^k})[x] / \langle x^{p^s}-1\rangle $. The {\em annihilator} of $C$, denoted by $\Ann (C)$, is defined to be the set $\{f(x)\in (\mathbb{F}_{p^k}+u\mathbb{F}_{p^k})[x] / \langle x^{p^s}-1\rangle \mid f(x)g(x)=0 \text{ for all } g(x) \in C \}$.
        \end{definition}
        The following properties of the annihilator can be obtained using arguments similar to those in  the case of Galois rings in \cite{HMK2008}.
        \begin{theorem}\label{annformideals}
            Let $C$ be an ideal of $(\mathbb{F}_{p^k}+u\mathbb{F}_{p^k})[x] / \langle x^{p^s}-1\rangle $. Then the following statements hold.
            \begin{enumerate}[$(i)$]
                \item  $\Ann(C)$ is an ideal of $(\mathbb{F}_{p^k}+u\mathbb{F}_{p^k})[x] / \langle x^{p^s}-1\rangle $.
                \item If $|C|=(p^k)^d$, then $|\Ann(C)| =(p^k)^{(2\cdot p^s-d)}$.
                \item    $\Ann (\Ann (C))  = C$ 
            \end{enumerate}
        \end{theorem}

        \begin{theorem}\label{A&A'} Let $\mathcal{I}$ denote  the set of  ideals of~$(\mathbb{F}_{p^k}+u\mathbb{F}_{p^k})[x] / \langle x^{p^s}-1\rangle $ and  let $\mathcal{A}=\{ C\in \mathcal{I}\mid T_0(C)+T_1(C)\leq p^s \}$ and $\mathcal{A}'=\{ C\in \mathcal{I}\mid T_0(C)+T_1(C)\geq p^s \}$. Then the map $\phi :\mathcal{A}\rightarrow \mathcal{A}'$ defined  by $C\mapsto \Ann (C)$ is a bijection.
        \end{theorem}
        
        The  rest of this section is devoted to the determination of  all ideals in  $(\mathbb{F}_{p^k}+u\mathbb{F}_{p^k})[x] / \langle x^{p^s}-1\rangle $. In view of Theorem \ref{A&A'}, it suffices to focus on the ideals in $\mathcal{A}$.
        
        For each  $C=\langle \langle f_0(x),f_1(x)\rangle \rangle$ in $\mathcal{A}$,  if $f_0(x)=0$, then $T_0(C)=p^s $ and  $T_1(C)=0$. Hence, the only ideal in $\mathcal{A}$ with $f_0(x)=0$ is of the form $\langle \langle 0,u\rangle \rangle$. In the following two theorems, we  assume that  $f_0(x)\neq 0$.
        \begin{theorem}\label{formofcycliccodes} Let 
            $\langle \langle (x-1)^{i_0}+u(x-1)^{t}h(x),u(x-1)^{i_1}\rangle \rangle$  be the representation of an ideal in $(\mathbb{F}_{p^k}+u\mathbb{F}_{p^k})[x] / \langle x^{p^s}-1\rangle $. Then it is a  representation of   an ideal in $\mathcal{A}$ if and only if $i_0,i_1,t$ are integers and $h(x)\in \mathbb{F}_{p^k}[x]/ \langle x^{p^s}-1 \rangle $ such that $0\leq i_0<p^s, $ $0\leq i_1 \leq \min\{ i_0, p^s-i_0\}, $ $t\geq 0, $ $t+\deg(h(x))<i_1$ and $h(x)$ is either zero or a unit in $\mathbb{F}_{p^k}[x] / \langle x^{p^s}-1\rangle $.
        \end{theorem}
        \begin{proof}  Form Theorem~\ref{formidealunique},  we have that   $i_0,i_1,t$ are integers and $h(x)\in \mathbb{F}_{p^k}[x]/ \langle x^{p^s}-1 \rangle $ such that $0\leq i_0<p^s, $ $0\leq i_1 \leq i_0,$ $t\geq 0,$ $t+\deg(h(x))<i_1$ and $h(x)$ is either zero or a unit in $\mathbb{F}_{p^k}[x] / \langle x^{p^s}-1\rangle $.
            
            Assume that $\langle \langle (x-1)^{i_0}+u(x-1)^{t}h(x),u(x-1)^{i_1}\rangle \rangle$ is a representation of an ideal in $\mathcal{A}$. Then $i_0+i_1\leq p^s$ which implies that  $i_1\leq p^s-i_0$. Hence,  we have  $i_1 \leq \min\{i_0,p^s-i_0\}.$

            Conversely, assume that $C=\langle (x-1)^{i_0}+u(x-1)^{t}h(x),u(x-1)^{i_1}\rangle $, where $i_0,i_1,t$ are integers and $h(x)\in \mathbb{F}_{p^k}[x]/ \langle x^{p^s}-1 \rangle $ such that $0\leq i_0<p^s, $ $0\leq i_1 \leq \min\{ i_0,p^s-i_0\}, $ $t\geq 0, $ $t+\deg(h(x))<i_1$ and $h(x)$ is either zero or a unit in $\mathbb{F}_{p^k}[x] / \langle x^{p^s}-1\rangle $. Clearly,  $i_0+i_1\leq p^s$. To show that $C=\langle \langle (x-1)^{i_0}+u(x-1)^{t}h(x),u(x-1)^{i_1}\rangle\rangle  $  in  $\mathcal{A}$ and it remains to proof that  $T_0(C)=i_0$ and $T_1(C)=i_1$. 
            
            Let $D=\langle (x-1)^{p^s-i_1}-u(x-1)^{p^s-i_0-i_1+t}h(x),u(x-1)^{p^s-i_0}\rangle $. It is not difficult to see that 
            \begin{align*}D=\langle (x-1)^{p^s-i_1}+u(x-1)^{p^s-i_0}-u(x-1)^{p^s-i_0-i_1+t}h(x),u(x-1)^{p^s-i_0}\rangle . 
            \end{align*}
            Since 
            \begin{align*}
            \left( (x-1)^{i_0}+u(x-1)^{i_1}h(x)\right) \left( (x-1)^{p^s-i_1}+u(x-1)^{p^s-i_0}-u(x-1)^{p^s-i_0-i_1+t}h(x)\right) &=0,\\\left((x-1)^{i_0}+u(x-1)^{i_1}h(x)\right)  \left(u(x-1)^{p^s-i_0}\right)&=0,
            \end{align*}
            and
            \[\left(u(x-1)^{i_1}\right) \left( (x-1)^{p^s-i_1}+u(x-1)^{p^s-i_0}-u(x-1)^{p^s-i_0-i_1+t}h(x)\right)=0, \]
            we have $D\subseteq \Ann (C).$ By Proposition~\ref{counttorandc}, we obtain that $T_j(C)\leq i_j$ and $T_j(D)\leq p^{s-i_{1-j}}$ for all $j\in \{0,1\}$. Hence, $|C| \geq (p^k)^{2\cdot p^s-i_0-i_1}$ and $|\Ann (C) | \geq |D| \geq (p^k)^{i_0+i_1}$. Since  $|C| \cdot | \Ann (C)| =(p^k)^{2\cdot p^s}$, we have   $\Ann (C)=D$. Therefore, $T_0(C)=i_0$, $T_1(C)=i_1$ as desired.
        \end{proof}
        
        Since every polynomial $\sum_{i=0}^{m}a_i(x-1)^i$ in $\mathbb{F}_{p^k}[x]$ is either $0$ or $(x-1)^t h(x)$, where $h(x)$ is  a unit in $\mathbb{F}_{p^k}[x]$ and $0\leq t \leq m-\deg(h(x))$, Theorem~\ref{formofcycliccodes} is rewritten as follows:
        \begin{theorem}\label{formidealexact} The expression 
            $\langle \langle (x-1)^{i_0}+u\displaystyle\sum_{j=0}^{i_1-1}h_j(x-1)^{j},u(x-1)^{i_1}\rangle \rangle$   represents an ideal in $\mathcal{A}$ if and only if $i_0$ and $i_1$   are integers such that $0\leq i_0<p^s$, $0\leq i_1 \leq \min \{i_0,p^s-i_0\}$,  $i_0+i_1\leq p^s$, and $h_j \in \mathbb{F}_{p^k}$ for all $0\leq j<i_1$.
        \end{theorem}
        The number of distinct ideals of~$(\mathbb{F}_{p^k}+u\mathbb{F}_{p^k})[x] /\langle x^{p^s}-1\rangle $ of a fixed $d=T_0+T_1$ is given in the following proposition.
        \begin{proposition} \label{countofcycliccodesfixi1} Let $0\leq d\leq p^s$.  Then the number of distinct ideals in 
            $(\mathbb{F}_{p^k}+u\mathbb{F}_{p^k})[x] /\langle x^{p^s}-1\rangle $ with $T_0+T_1=d$ is 
            \[\frac{p^{k\left( K +1\right)}-1}{p^k-1}, \]
            where $K= \min\{\lfloor \frac{d}{2} \rfloor ,p^s-\lfloor \frac{d}{2} \rfloor \}$.
        \end{proposition}
        \begin{proof}
            Let $T_1 = i_1 $ and $i_0 :=T_0=d-T_1 $ be fixed.
            
            \noindent $\mathbf{Case~1:}$ $d< p^s$. Then $i_0\leq i_0+i_1=T_0+T_1=d< p^s.$
            By Theorem~\ref{formidealexact}, it follows that $C=\langle \langle  (x-1)^{i_{0}}+u\displaystyle\sum_{j=0}^{i_{1}-1}h_{j}(x-1)^j,u(x-1)^{i_{1}}\rangle \rangle $. Then the choices for  $\displaystyle\sum_{j=0}^{i_1-1}h_j(x-1)^j$ is $(p^k)^{i_1}$. By Theorem~\ref{formidealexact} again, we also have $T_1\leq \min\{T_0,p^s-T_0\}$. Since $T_0+T_1=d$, we obtain that $T_1\leq \lfloor \frac{d}{2} \rfloor \leq T_0$, an hence,  $T_1\leq \min\{ \lfloor \frac{d}{2} \rfloor ,p^s-T_0\}\leq \min\{ \lfloor \frac{d}{2} \rfloor ,p^s-\lfloor \frac{d}{2} \rfloor\} $. Now, vary $T_1$ from $0$ to $K$, we obtain that there are $1+p^k+\ldots +(p^k)^K=\frac{p^{k(K+1)}-1}{p^k-1}$ ideals with $T_0+T_1=d$.
            
            \noindent 	$\mathbf{Case~2:}$ $d=p^s$. If $i_0=p^s$, then the only ideal with $T_0+T_1=p^s$ is the ideal represented by $\langle \langle 0,u\rangle \rangle$. If $i_0<p^s$, then we have $p^k+(p^k)^2\ldots +(p^k)^K$ ideals  by   arguments similar to those in  Case 1.
        \end{proof}
        
        For a cyclic code $C$ in $\mathcal{A}$, we have $C\ne \Ann(C)$ whenever $T_0(C)+T_1(C)<p^s$. In the case where $T_0(C)+T_1(C)=p^s$, 
        by the proof of Theorem \ref{formofcycliccodes},  the annihilator of the  cyclic code $C=\langle \langle (x-1)^{i_0}+u(x-1)^{t}h(x),u(x-1)^{i_1}\rangle \rangle$ is of the form $\Ann(C)=\langle (x-1)^{i_0}-u(x-1)^{t}h(x),u(x-1)^{i_1}\rangle$.  If $p$ is odd, then $C=\Ann(C)$ occurs only the case $h(x)=0$.  In the case where $p=2$, $C=\Ann(C)$ is always true.  By Proposition \ref{countofcycliccodesfixi1} and  the bijection given in Theorem \ref{A&A'}, the number of cyclic  codes of length $p^s$ over $\mathbb{F}_{p^k}+u\mathbb{F}_{p^k}$ can be summarized as follows.
        \begin{corollary}
            \label{cor-cyclic-summ}
            The number of cyclic codes of length $p^s$ over $\mathbb{F}_{p^k}+u\mathbb{F}_{p^k}$ is 
            \[N(p^k,p^s)=\begin{cases}
            2\left(\sum\limits_{d=0}^{p^s} \frac{p^{k( \min\{\lfloor \frac{d}{2}\rfloor, p^{s}-\lfloor \frac{d}{2}\rfloor\} +1)}-1}{p^k-1}\right) -\frac{p^{k( p^{s-1} +1)}-1}{p^k-1} &\text{ if } p=2,\\
            2\left(\sum\limits_{d=0}^{p^s} \frac{p^{k( \min\{\lfloor \frac{d}{2}\rfloor, p^{s}-\lfloor \frac{d}{2}\rfloor\} +1)}-1}{p^k-1}\right) -1  &\text{ if } p \text{ is odd}.
            \end{cases}\]
        \end{corollary}
        \begin{proof}
            From  Theorem \ref{A&A'}, the number of cyclic  codes of length $p^s$ over $\mathbb{F}_{p^k}+u\mathbb{F}_{p^k}$  is $|\mathcal{A}\cup \mathcal{A}^\prime|=|\mathcal{A}|+|\mathcal{A}^\prime|-|\mathcal{A}\cap\mathcal{A}^\prime|$.  The desired results follow immediately form the discussion above. 
        \end{proof}
        
        \section{Self-Dual Cyclic Codes of Length $p^s$ over $\mathbb{F}_{p^k}+u\mathbb{F}_{p^k}$}
        
        In this section, characterization and enumeration self-dual cyclic codes of length $p^s$ over $\mathbb{F}_{p^k}+u\mathbb{F}_{p^k}$ are given under the Euclidean   and Hermitian inner products.

        \subsection{Euclidean Self-Dual Cyclic Codes of Length $p^s$ over $\mathbb{F}_{p^k}+u\mathbb{F}_{p^k}$}
        Characterization and enumeration of Euclidean  self-dual cyclic codes of length $p^s$ over $\mathbb{F}_{p^k}+u\mathbb{F}_{p^k}$ are given in this subsection.

        For each subset $A$ of $(\mathbb{F}_{p^k}+u\mathbb{F}_{p^k})[x] /\langle x^{p^s}-1\rangle $, denote by $\overline{A}$ the set of polynomials $\overline{f(x)}$ for all $f(x)$ in $A$, where \ $\bar{ }$ \  is viewed as  the conjugation on  the group ring $(\mathbb{F}_{p^k}+u\mathbb{F}_{p^k})[\mathbb{Z}_{p^s}]  $ defined in Section $2$. 
        From the definition of the annihilator,     the next theorem  can be derived similar to \cite[Proposition 2.12]{DL2004}.
        
        \begin{theorem}\label{annC}
            Let $C$ be an ideal in $(\mathbb{F}_{p^k}+u\mathbb{F}_{p^k})[x] /\langle x^{p^s}-1\rangle $. Then $C^{\perp_{\rm E}}=\overline{\Ann(C)}$.
        \end{theorem}

        Using the unique generators of an ideal $C$ in  $(\mathbb{F}_{p^k}+u\mathbb{F}_{p^k})[x] /\langle x^{p^s}-1\rangle $ determined in  Theorem~\ref{formidealexact}, the Euclidean dual of $C$ can be  given in the following theorem.
        \begin{theorem}\label{formeesdcycliccodes}
            Let $C=\left\langle \left\langle  (x-1)^{i_{0}}+u\displaystyle\sum_{j=0}^{i_{1}-1}h_{j}(x-1)^j,u(x-1)^{i_{1}}\right\rangle\right \rangle $ be an ideal in $\mathcal{A}$, where $h_{j} \in \mathbb{F}_{p^k}$ for all $0\leq j \leq i_i-1$. Then $C^{\bot_\mathrm{E}}$  is in the form of

            \begin{align}\label{eq-eucliddual} C^{\bot_\mathrm{E}}&=\left\langle \left\langle (x-1)^{p^s -i_{1}}-u(x-1)^{p^s -i_{0}-i_{1}}\left(\displaystyle\sum_{r=0}^{i_{1}-1}\displaystyle\sum_{j=0}^{r} (-1)^{i_{0}+j} \binom{i_{0}-j}{r-j}h_{j}(x-1)^r\right),\right.\right. \notag \\
            &\quad\quad \quad \left. \left.u(x-1)^{p^s-i_{0}}\right\rangle \right\rangle. 
            \end{align}
        \end{theorem}

        \begin{proof}
            From the proof of Theorem \ref{formofcycliccodes}, we have 
            \[\Ann (C)=\left\langle \left\langle  (x-1)^{p^{s}-i_{1}}-u(x-1)^{p^s-i_{0}-i_{1}}\displaystyle\sum_{j=0}^{i_{1}-1}h_{j}(x-1)^j,u(x-1)^{p^s -i_{0}}\right\rangle \right\rangle .\]
            By Theorem~\ref{annC}, it follows that  $C^{\bot_\mathrm{E}}=\overline{\Ann (C)}$. Hence,
            $C^{\bot_\mathrm{E}}$ contains the elements $u(x-1)^{p^s-i_{0}}$ and 
            \[(x-1)^{p^{s}-i_{1}}-u(x-1)^{p^s-i_{0}-i_{1}}\displaystyle\sum_{j=0}^{i_{1}-1}(-1)^{i_{0}+j}h_{j}(x-1)^jx^{i_{0}-j}.\]
            By writing $x=(x-1)+1$ and using the Binomial Theorem,  it follows that  $C^{\bot_\mathrm{E}}$ contains the element
            \[(x-1)^{p^{s}-i_{1}}-u(x-1)^{p^s-i_{0}-i_{1}}\left(\displaystyle\sum_{j=0}^{i_{1}-1}\left(\displaystyle\sum_{l=0}^{i_{0}-j}(-1)^{i_{0}+j}\binom{i_{0}-j}{l}h_{j}(x-1)^{i_0 -l}\right)\right).\]
            Hence, 
            \begin{align*} &\left\langle \left\langle (x-1)^{p^{s}-i_{1}}-u(x-1)^{p^s-i_{0}-i_{1}}\left(\displaystyle\sum_{j=0}^{i_{1}-1}\left(\displaystyle\sum_{l=0}^{i_{0}-j}(-1)^{i_{0}+j}\binom{i_{0}-j}{l}h_{j}(x-1)^{i_0 -l}\right)\right),\right.\right.\\
            &\left.\left.\quad\quad\quad  (x-1)^{p^s-i_{1}}\right\rangle \right\rangle\subseteq C^{\bot_\mathrm{E}}.
            \end{align*}
            By counting the number of elements,  the two sets are equal as desired.
            Updating the indices,  it can be concluded that 
            \begin{align*}
            C^{\bot_\mathrm{E}}&=\left\langle \left\langle (x-1)^{p^{s}-i_{1}}-u(x-1)^{p^s-i_{0}-i_{1}}\left(\displaystyle\sum_{r=0}^{i_{1}-1}\left(\displaystyle\sum_{j=0}^{r}(-1)^{i_{0}+j}\binom{i_{0}-j}{r-j}h_{j}(x-1)^{r}\right)\right), \right.\right.\\
            &\quad\quad\quad \left.\left. (x-1)^{p^s-i_{1}} \right\rangle \right\rangle.
            \end{align*}
        \end{proof}
        
        Assume that an ideal $C=\left\langle \left\langle  (x-1)^{i_{0}}+u\displaystyle\sum_{j=0}^{i_{1}-1}h_{j}(x-1)^j,u(x-1)^{i_{1}}\right\rangle\right \rangle $  in  $(\mathbb{F}_{p^k}+u\mathbb{F}_{p^k})[x] / \langle x^{p^s}-1\rangle $  is Euclidean self-dual, i.e., $C=C^{\bot {\rm E}}$.  By Theorem \ref{eq-eucliddual}, it is equivalent to  $p^s=i_0+i_1$ and 
        \begin{align}\label{h_j}- h_t=\displaystyle\sum_{j=0}^{t} (-1)^{i_{0}+l} \binom{i_{0}-j}{r-j}h_{j}
        \end{align}
        in $\mathbb{F}_{p^k}$ for all $0\leq t\leq i_i-1$.
        
        We note that, if $i_1=0$, then it is not difficult to see  that  only the ideal generated by $u$ is Euclidean  self-dual.  
        
        For the case $i_1\geq 1$,  the situation is more complicated. First, we recall  an $i_1\times i_1$ matrix  $M(p^s,i_1)$  over $\mathbb{F}_{p^k}$ defined in  \cite{HMK2012} as  
        \begin{equation}\label{matrixm}
        M(p^s,i_1)=\begin{bmatrix}
        (-1)^{i_0}+1	&0  & 0 & \ldots & 0 \\ 
        (-1)^{i_0}\binom{i_0}{1}	&  (-1)^{i_0+1}+1& 0  & \ldots  & 0  \\ 
        (-1)^{i_0}\binom{i_0}{2}	& (-1)^{i_0+1}\binom{i_0-1}{1} &  (-1)^{i_0+2}+1 &\ldots & 0  \\
        \vdots	& \vdots & \vdots  &\ddots & \vdots  \\ 
        (-1)^{i_0}\binom{i_0}{i_1-1}	& (-1)^{i_0+1}\binom{i_0-1}{i_1-2} &  (-1)^{i_0+2}\binom{i_0-2}{i_1-3} &\ldots & (-1)^{i_0+i_1-1}+1
        \end{bmatrix}. 
        \end{equation}
        
        It is not difficult to see that  $i_1$ equations from \eqref{h_j}   are  equivalent to the matrix equation
        \begin{equation}\label{solveesd}
        M(p^s,i_1)\textbf{h}=\textbf{0}
        \end{equation}
        where $\textbf{h}=(h_0,h_1,\ldots,h_{i_1-1})^T$ and $\textbf{0}=(0,0,\ldots,0)^T$.
        
        Moreover,  it can be concluded that the ideal $C$ is Euclidean self-dual if and only if $p^s=i_0+i_1$ and $h_0,h_1,\ldots ,h_{i_1-1}$ satisfy \eqref{solveesd}. Since $h_0=h_1=\dots =h_{i_1-1}=0$ is a solution of~\ref{solveesd},  the corresponding idea  $\langle \langle (x-1)^{p^s-i_{1}},u(x-1)^{i_{1}}\rangle \rangle $ is Euclidean self-dual in $(\mathbb{F}_{p^k}+u\mathbb{F}_{p^k})[x]/ \langle x^{p^s}-1\rangle $. Hence, for a fixed  first torsion degree $1\leq i_1\leq p^s$,  a Euclidean self-dual ideal in  $(\mathbb{F}_{p^k}+u\mathbb{F}_{p^k})[x]/ \langle x^{p^s}-1\rangle $ always exists. 
        By solving  \eqref{solveesd}, all Euclidean self-dual ideals in  $(\mathbb{F}_{p^k}+u\mathbb{F}_{p^k})[x] / \langle x^{p^s}-1\rangle $  can be constructed. Therefore, for a fixed  first torsion degree $1\leq i_1\leq p^s$, the number of Euclidean self-dual ideals in  $(\mathbb{F}_{p^k}+u\mathbb{F}_{p^k})[x] / \langle x^{p^s}-1\rangle $  equals the number of solutions of \eqref{solveesd} which is $p^{k\kappa}$, where $\kappa$ is the nullity of $M(p^s,i_1)$ determined in  \cite{HMK2012}.
        \begin{proposition}[{\cite[Proposition 3.3]{HMK2012}}]\label{numberofnullity}
            Let $\kappa$ be the nullity of $M(p^s,i_1)$. Then 
            \[   \kappa = \begin{cases}
            \lfloor \frac{i_1}{2}\rfloor  & \text{ if  }p \text{ is odd};\\
            \lceil \frac{i_1+1}{2} \rceil & \text{ if }p=2.
            \end{cases}    \]
        \end{proposition}
        
        The number of Euclidean self-dual cyclic codes  in  $(\mathbb{F}_{p^k}+u\mathbb{F}_{p^k})[x] /\langle x^{p^s}-1\rangle $  with first torsional degree~$i_1$  is given in terms of the nullity of  $M(p^s,i_1)$ as follows.

        \begin{proposition}\label{countesdcycliccodesfixi1}
            Let $i_1>0$ and let $\kappa$ be the nullity of $M(p^s,i_1)$ over $\mathbb{F}_{p^k}$.  Then  the number of Euclidean self-dual cyclic codes of length $p^s$ over $\mathbb{F}_{p^k}$ with first torsional degree~$i_1$  is 
            \[(p^k)^\kappa . \]
        \end{proposition}

        From Theorem \ref{formofcycliccodes},  we have  $0\leq i_1\leq \lfloor \frac{p^s}{2}\rfloor$ since $i_0+i_1=p^s$.    Hence,  the  number of Euclidean self-dual cyclic codes of  length $p^s$ over $\mathbb{F}_{p^k}+u\mathbb{F}_{p^k}$ is given by the following corollary.

        \begin{corollary}\label{countesdcycliccodes} Let $p$ be a prime and let $s$ and $k$ be positive integers.  Then the  following statements hold.
            \begin{enumerate}[$(i)$]
                \item If  $p$ is odd, then the number of Euclidean self-dual cyclic codes of length $p^s$ over $\mathbb{F}_{p^k}+u\mathbb{F}_{p^k}$ is
                \begin{align*}
                NE(p^k,p^s)=\begin{cases}
                2\left(\frac{(p^k)^{\frac{p^{s}+1}{4}}-1}{p^k-1}\right)& \text{ if } p^s \equiv 3\,{\rm mod }\, 4,\\
                2\left(\frac{(p^k)^{\frac{p^{s}-1}{4}}-1}{p^k-1}\right)+(p^k)^\frac{p^s-1}{4}& \text{ if } p^s \equiv 1\,{\rm mod }\, 4.
                \end{cases}
                \end{align*}
                \item If  $p=2$, then the number of Euclidean self-dual cyclic codes of length $2^s$ over $\mathbb{F}_{2^k}+u\mathbb{F}_{2^k}$ is
                \begin{align*}
                NE(p^k,p^s)=\begin{cases}
                1+2^k  & \text{ if  }s=1,\\
                1+2^k+(2^k)^2 & \text{ if }s=2,\\
                1+2^k+2(2^k)^2\left( \frac{(2^k)^{(2^{s-2}-1)}-1}{2^k-1}\right) & \text{ if  }s\geq 3 .
                \end{cases}  
                \end{align*}
            \end{enumerate}
        \end{corollary}
        \begin{proof}
            From Propositions   \ref{numberofnullity} and   \ref{countesdcycliccodesfixi1},   the number of Euclidean self-dual cyclic codes of length $2^s$ over $\mathbb{F}_{2^k}+u\mathbb{F}_{2^k}$ is $\sum\limits_{i_1=0}^{\lfloor \frac{p^s}{2}\rfloor} (p^k)^{\lfloor \frac{i_1}{2}\rfloor}$.  Apply a suitable geometric sum, the results follow. 
        \end{proof}

        \subsection{Hermitian Self-Dual Cyclic Codes of Length $p^s$ over $\mathbb{F}_{p^k}+u\mathbb{F}_{p^k}$}
        Under the assumption  that $k$ is even, characterization and enumeration Hermitian self-dual cyclic codes of length $p^s$ over $\mathbb{F}_{p^k}+u\mathbb{F}_{p^k}$ are given in this section.

        For a subset $A$ of $(\mathbb{F}_{p^k}+u\mathbb{F}_{p^k})[x] / \langle x^{p^s}-1\rangle $, let \[ \rho(A):=\left\{\left.\displaystyle\sum_{i=0}^{p^s-1}\rho({a_i})x^i \right\vert \displaystyle\sum_{i=0}^{p^s-1}a_ix^i \in A\right\},\]
        where $\rho(a+ub)=a^{p^{\frac{k}{2}}}+ub^{p^{\frac{k}{2}}}$.
        
        Based on the structural characterization of $C$ given in Theorem~\ref{formidealexact}, the  Hermitian dual of $C$  is determined as follows.
        \begin{theorem}\label{formehsdcycliccodes}
            Let $C$ be an ideal in $\mathcal{A}$ and \[C=\left\langle \left\langle  (x-1)^{i_{0}}+u\displaystyle\sum_{j=0}^{i_{1}-1}h_{j}(x-1)^j,u(x-1)^{i_{1}}\right\rangle \right\rangle,\]  where $h_{j} \in \mathbb{F}_{p^k}$. Then $C^{\bot_\mathrm{H}}$ has the representation
            \begin{align*}  C^{\bot_\mathrm{H}}&=\left\langle \left\langle  (x-1)^{p^s -i_{1}}-u(x-1)^{p^s -i_{0}-i_{1}}\left(\displaystyle\sum_{r=0}^{i_{1}-1}\displaystyle\sum_{j=0}^{r} (-1)^{i_{0}+j} \binom{i_{0}-j}{r-j}h_{j}^{p^{k/2}}(x-1)^r\right),\right.
            \right.\\
            &\quad \quad \quad \left.\left. u(x-1)^{p^s-i_{0}}\right\rangle \right\rangle . 
            \end{align*}
        \end{theorem}
        \begin{proof}
            From Theorem~\ref{formeesdcycliccodes} and the fact that $C^{\bot_\mathrm{H}}=\rho({C^{\perp_{\rm E}}})$, the result follows..
        \end{proof}
        
        
        Assume that $C$ is Hermitian self-dual. Then  $C=C^{\bot_{\rm H}}$ which implies that $| C |=(p^k)^{p^s}$  and  $i_0+i_1=p^s$. 
        
        If $i_1=0$, then it is not difficult to see that the ideal generated by $u$  is only  Hermitian self-dual cyclic code   of length $p^s$ over $\mathbb{F}_{p^k}+u\mathbb{F}_{p^k}$. 
        
        Assume that $i_1\geq 1$.  Then \[-uh_t^{p^{k/2}}=u\displaystyle\sum_{j=0}^{t} (-1)^{i_{0}+j}\binom{i_{0}-j}{t-j}h_{j}\]
        for all $0\leq t\leq i_1-1.$
        
        From $i_i$ equations above and the definition of   $M(p^s,i_1)$, we have 
        \begin{equation}\label{formmatrix}
        M(p^s,i_1)\textbf{h}+(\textbf{h}^{{p}^{{k/2}}}-\textbf{h})=\textbf{0}
        \end{equation}
        where $\textbf{h}=(h_1,h_2,\ldots ,h_{i_1})$, $\textbf{h}^{\textbf{p}^{\textbf{k/2}}}=(h_1^{p^{k/2}},h_2^{p^{k/2}},\ldots ,h_{i_1}^{p^{k/2}})$ and $\textbf{0}=(0,0,\ldots ,0)$.
        
        From Theorem \ref{formofcycliccodes},  we have  $0\leq i_1\leq \lfloor \frac{p^s}{2}\rfloor$ since $i_0+i_1=p^s$.  
        
        \begin{proposition}\label{countesdcycliccoddsfixi1}
            Let $k$ be an even positive integer and let $i_1$ be a positive integer such that $i_1\leq \lfloor \frac{p^{k}}{2}\rfloor$. Then the number of solution of $(\ref{formmatrix})$ in $\mathbb{F}_{p^k} ^{i_1}$ is
            \[p^{ki_1/2}.\]
        \end{proposition}
        \begin{proof}   
            Let  $\Psi : \mathbb{F}_{p^k}\rightarrow \mathbb{F}_{p^k}$ defined by $\alpha \mapsto \alpha^{p^{k/2}}- \alpha $ for all $\alpha \in \mathbb{F}_{p^k}$.
            Using the fact that  $\Psi(1)=0=\Psi(0)$   and arguments similar to those in \cite[Proposition 3.3]{S2015}, the result follows.
        \end{proof}

        For a prime number $p$, a  positive integer $s$ and  an even positive integer $k$, the number of Hermitian self-dual cyclic codes of length $p^s$ over $\mathbb{F}_{p^k}+u\mathbb{F}_{p^k}$ can be determined  in the following corollary.
        \begin{corollary}\label{counthsdcycliccodes} Let $p$ be a prime and let $s$ and $k$ be positive integers such that $k$ is even. 
            Then the number of Hermitian self-dual cyclic codes of length $p^s$ over $\mathbb{F}_{p^k}+u\mathbb{F}_{p^k}$ is
            \[NH(p^k,p^s)=\displaystyle\sum_{i_1=0}^{\lfloor \frac{p^s}{2} \rfloor}p^{ki_1/2}=\frac{(p^{k/2})^{\lfloor \frac{p^s}{2}\rfloor +1}-1}{p^{k/2}-1}.\]
        \end{corollary}
        \section{Conclusions and Remarks}

        
        Euclidean and Hermitian self-dual abelian codes in non-PIGAs $\mathbb{F}_{2^k}[A\times \mathbb{Z}_2\times \mathbb{Z}_{2^s}]$ are studied.  The  complete characterization and enumeration of such abelian codes are given and  summarized as follows.
        
        In Corollaries \ref{selfDuA} and \ref{selfDuA2},  self-dual  abelian code in  $\mathbb{F}_{2^k}[A\times \mathbb{Z}_2\times \mathbb{Z}_{2^s}]$ are shown to be a suitable Cartesian product of cyclic codes, Euclidean self-dual cyclic codes, and Hermitian self-dual cyclic codes of length $2^s$ over some Galois extension of the ring $\mathbb{F}_{2^k}+u\mathbb{F}_{2^k}$.   Subsequently, the characterizations and enumerations of  cyclic and self-dual cyclic codes of length $p^s$ over   $\mathbb{F}_{p^k}+u\mathbb{F}_{p^k}$ are studied for all primes $p$. Combining these results,  the following  enumerations of  Euclidean and Hermitian self-dual  abelian codes in $\mathbb{F}_{2^k}[A\times \mathbb{Z}_2\times \mathbb{Z}_{2^s}]$  are rewarded.

        For each abelian group $A$ of odd order and  positive integers $s$ and $k$, the number of Euclidean  self-dual abelian codes in $\mathbb{F}_{2^k}[A\times \mathbb{Z}_2\times \mathbb{Z}_{2^s}]$   is given in Theorem \ref{NEA} in   terms of the numbers $N$, $NE$, and $NH$   of  cyclic codes, Euclidean self-dual cyclic codes, and Hermitian self-dual cyclic codes of length $2^s$ over a Galois extension of  $\mathbb{F}_{2^k}+u\mathbb{F}_{2^k}$, respectively.

        In addition, if  $k$ is even, the number of Hermitian  self-dual abelian codes in $\mathbb{F}_{2^k}[A\times \mathbb{Z}_2\times \mathbb{Z}_{2^s}]$   is given in Theorem \ref{NHA} in   terms of  the numbers $N$ and $NH$ of  cyclic codes  and Hermitian self-dual cyclic codes of length $2^s$ over a Galois extension of $\mathbb{F}_{2^k}+u\mathbb{F}_{2^k}$, respectively.

        We note that all   numbers  $N $, $NE $, and $NH $   are determined in Corollaries \ref{cor-cyclic-summ}, \ref{countesdcycliccodes}, and \ref{counthsdcycliccodes}, respectively. Therefore, the complete enumerations of     Euclidean and Hermitian self-dual abelian codes in $\mathbb{F}_{2^k}[A\times \mathbb{Z}_2\times \mathbb{Z}_{2^s}]$ are established.

        One of the interesting problems concerning the   enumeration  of  self-dual abelian codes in   $\mathbb{F}_2^k[A\times B]$,  where $A$ is an abelian group of odd order,  is the case where $B$ is a $2$-group  of other types.


 \section*{Acknowledgements}
 This work   is supported by the Thailand Research Fund  under Research Grant TRG5780065 and the DPST Research Grant 005/2557.

\end{document}